\documentclass[12pt]{amsart}
\usepackage{amsmath}
\usepackage{amssymb}
\usepackage{amsthm}
\usepackage{bm}

\numberwithin{equation}{section}
\newtheorem{theorem}{Theorem}[section]
\newtheorem{lemma}[theorem]{Lemma}
\newtheorem{proposition}[theorem]{Proposition}

\newtheorem{definition}[theorem]{Definition}

\newtheorem{mthm}{Main Theorem}

\theoremstyle{definition}
\newtheorem{remark}[theorem]{Remark}
\newtheorem{example}[theorem]{Example}

\allowdisplaybreaks[3]

\begin{document}
\title[Zeta distributions generated by Dirichlet series and infinite divisibility]{Zeta distributions generated by Dirichlet series and their (quasi) infinite divisibility}
\author[T.~Nakamura]{Takashi Nakamura}
\address[T.~Nakamura]{Department of Liberal Arts, Faculty of Science and Technology, Tokyo University of Science, 2641 Yamazaki, Noda-shi, Chiba-ken, 278-8510, Japan}
\email{nakamuratakashi@rs.tus.ac.jp}
\urladdr{https://sites.google.com/site/takashinakamurazeta/}
\maketitle

\begin{abstract}
Let $a(1) >0$, $a(n) \ge 0$ for $n \ge 2$ and $a(n) = O(n^\varepsilon)$ for any $\varepsilon >0$, and put $Z(\sigma + it):= \sum_{n=1}^\infty a(n) n^{-\sigma - it}$ where $\sigma , t \in {\mathbb{R}}$.
In the present paper, we show that any zeta distribution whose characteristic function is defined by ${\mathcal{Z}}_\sigma (t) :=Z(\sigma + it)/Z(\sigma)$ is pretended infinitely divisible if $\sigma >1$ is sufficiently large. Moreover, we prove that if ${\mathcal{Z}}_\sigma (t)$ is an infinitely divisible characteristic function for some $\sigma_{id} >1$, then ${\mathcal{Z}}_\sigma (t)$ is infinitely divisible for all $\sigma >1$. Note that the corresponding L\'evy or quasi-L\'evy measure can be given explicitly. A key of the proof is a corrected version of Theorem 11.14 in Apostol's famous textbook. 
\end{abstract}

%%%%%%%%%%%%%%%%%%%%%%
\section{Introduction}
%%%%%%%%%%%%%%%%%%%%%%

%%%%%%%%%%%%%%%%%%%%%%%%%%%%%%%%%%%%%%%%%%%%%%%%%%%%%
\subsection{Infinite and quasi infinite divisibility}
%%%%%%%%%%%%%%%%%%%%%%%%%%%%%%%%%%%%%%%%%%%%%%%%%%%%%
Let $\widehat{\mu}(t):=\int_{\mathbb{R}}e^{{\rm i} tx}\mu (dx),\, t \in {\mathbb{R}}$ be the characteristic function of a distribution $\mu$, as usual. The following L\'evy--Khintchine representation is well-known (see for example \cite[(8.5)]{S99}). If a distribution $\mu$ is infinitely divisible, then it holds that
\begin{equation}\label{INF}
\widehat{\mu}(t) = \exp\left[-\frac{a}{2} t^2 + {\rm i} \gamma t +
\int_{\mathbb{R}} \left(e^{{\rm i} tx }-1-
\frac{{\rm i} tx}{1+|x|^2}\right)\nu(dx)\right],\quad t \in {\mathbb{R}},
\end{equation}
where $a\ge 0$, $\nu$ is a measure on ${\mathbb{R}}$ satisfying $\nu(\{0\}) = 0$ and $\int_{\mathbb{R}} (|x|^{2} \wedge 1) \nu(dx) < \infty$ and $\gamma \in {\mathbb{R}}$. Furthermore, the representation of $\widehat{\mu}$ in (\ref{INF}) by $a$, $\nu$ and $\gamma$ is unique. Note that if the L\'evy measure $\nu$ in $\eqref{INF}$ satisfies the condition $\int_{|x|<1}|x|\nu(dx)<\infty$, then the representation $\eqref{INF}$ can be expressed as
\begin{equation}\label{INF2}
\widehat{\mu}(t) = \exp\left[-\frac{a}{2} t^2 +{\rm i} \gamma_0 t +
\int_{\mathbb{R}} \left(e^{{\rm i} tx }-1\right)\nu(dx)\right],\quad t \in {\mathbb{R}},
\end{equation}
where $\gamma_0=\gamma-\int_{{\mathbb{R}}}x\left(1+|x|^2\right)^{-1}\nu(dx)$ (see for instance \cite[(8.7)]{S99}).

There are characteristic functions expressed as quotients of two infinitely divisible characteristic functions. That class is called class of {\it quasi infinitely divisible distributions} and is defined as follows.

\begin{definition}[Quasi infinitely divisible distribution]
A distribution $\mu$ on ${\mathbb{R}}$ is called {\it quasi infinitely divisible} if it has a form of \eqref{INF} and the corresponding measure $\nu$ is a signed measure on ${\mathbb{R}}$ with total variation measure $|\nu|$ satisfying $\nu(\{0\}) =0$ and $\int_{\mathbb{R}} (|x|^{2} \wedge 1) |\nu|(dx) < \infty$.
\end{definition}

It should be mentioned that the triplet $(a,\nu ,\gamma)$ in this case is also unique if each component exists and that infinitely divisible distributions on ${\mathbb{R}}$ are quasi infinitely divisible if and only if the negative part of $\nu$ in the Jordan decomposition equals zero.
The measure $\nu$ is called {\it quasi}-L\'evy measure which has already appeared in some old books, for example,  Gnedenko and Kolmogorov \cite[p.~81]{GK68} and Linnik and Ostrovskii \cite[Chap.~6, \S 7]{LO} (see also \cite[Section 2.4]{Sato12}). Recently, Lindner, Pan  and Sato \cite[Theorem 4.1]{LS17} proved that the set of quasi infinitely divisible distributions on ${\mathbb{R}}$ is dense in the set of all probability distributions on ${\mathbb{R}}$ with respect to weak convergence. As an analogue of quasi infinitely divisible distribution, the following (quasi-infinitely divisible distributions without the condition $\int_{\mathbb{R}} (|x|^{2} \wedge 1) |\nu|(dx) < \infty$) is introduce in \cite[Section 1.3]{NaB}.
\begin{definition}[Pretended infinitely divisible distribution] \label{def:PID}
A distribution $\mu$ on ${\mathbb{R}}$ is called {\it pretended infinitely divisible} if it has a form of \eqref{INF} and the corresponding measure $\nu$ is a signed measure on ${\mathbb{R}}$ with $\nu(\{0\}) =0$.
\end{definition}

%%%%%%%%%%%%%%%%%%%%%%%%%%%%%%%%%%%%%%%%%%%%%
\subsection{Zeta functions and distributions}
%%%%%%%%%%%%%%%%%%%%%%%%%%%%%%%%%%%%%%%%%%%%%
Let $\{ a(n) \}$ be a sequence. A series of the form
\begin{equation}\label{eq:dergd1}
Z(s) := \sum_{n=1}^\infty \frac{a(n)}{n^s}, \qquad s = \sigma+ {\rm{i}} t, \quad \sigma,t \in {\mathbb{R}}
\end{equation}
is called a Dirichlet series. When $a(n)=1$ for all $n \in {\mathbb{N}}$, the Dirichlet series
\[
\zeta (s) = \sum_{n=1}^\infty \frac{1}{n^s} = \prod_p \Bigl( 1 - \frac{1}{p^s} \Bigr)^{-1},
\]
where the letter $p$ is a prime number, and the product of $\prod_p$ is taken over all primes, is the Riemann zeta function. Note that both the infinite series and the infinite product, so-called Euler product, converge absolutely when $\sigma >1$. We can find some general properties of Dirichlet series in \cite[Chapter 11]{Apo}.

Let $\sigma >1$ and put 
\[
f_{\sigma}(t):= \frac{\zeta (\sigma +{\rm i}t)}{\zeta (\sigma)}, \qquad t \in {\mathbb{R}}, 
\]
then $f_{\sigma}(t)$ is a characteristic function. Note that first it is appeared in \cite{Khi}. A distribution $\mu_{\sigma}$ on ${\mathbb{R}}$ is said to be a Riemann zeta distribution with parameter $\sigma$ if it has $f_{\sigma}(t)$ as its characteristic function. By the Euler product, the Riemann zeta distribution $\mu_{\sigma}$ is infinitely divisible and its L\'evy measures can be written the form as follows.

\begin{proposition}[see {\cite[p.~76]{GK68}}]\label{pro:RD}
The Riemann zeta distribution $\mu_{\sigma}$ is compound Poisson on ${\mathbb{R}}$ and one has
\[
f_{\sigma}(t) = \exp \left[ \int_{\mathbb{R}} \left(e^{{\rm i}tx}-1\right) {\mathfrak{N}}_{\sigma}(dx) \right], 
\qquad {\mathfrak{N}}_{\sigma}(dx):= \sum_{p} \sum_{r=1}^{\infty} \frac{p^{-r\sigma}}{r} \delta_{-r\log p}(dx),
\]
where $\delta_x$ is the delta measure at $x$.
\end{proposition}
Generalizations of this proposition can be found in Lin \& Hu \cite{Lin}, Aoyama \& Nakamura \cite{ANPE} and Nakamura \cite{NC}. The Riemann zeta function plays important role in the study of quasi or pretended infinitely divisible distributions. Let $\Xi_\sigma (t) := \xi (\sigma - {\rm{i}}t) / \xi (\sigma)$, where $\xi$ is the complete Riemann zeta function defined by 
\[
\xi (\sigma - {\rm{i}}t):= (\sigma - {\rm{i}}t)(\sigma -1 - {\rm{i}}t)\pi^{({\rm{i}}t-\sigma)/2} \Gamma \bigl((\sigma-{\rm{i}}t)/2 \bigr) \zeta(\sigma-{\rm{i}}t).
\]
Then we have the following which gives some connections between the Riemann hypothesis and pretended infinite divisibility (see Definition \ref{def:PID}). 
\begin{proposition}[{see \cite[Theorems 1.3 and 1.4]{NaB}}]\label{pro:BRH}
The Riemann hypothesis is true if and only if $\Xi_\sigma(t)$ is a pretended infinitely divisible characteristic function for any $1/2 < \sigma <1$. Moreover, $\Xi_\sigma(t)$ is a pretended-infinitely divisible characteristic function when $\sigma =1$, and $\Xi_\sigma(t)$ is not infinitely divisible but quasi-infinitely divisible for any $\sigma >1$.
\end{proposition}

%%%%%%%%%%%%%%%%%%%%%%%%%%%%%%%%%%%%%%%%%%%%%%%
\subsection{Main results}
%%%%%%%%%%%%%%%%%%%%%%%%%%%%%%%%%%%%%%%%%%%%%%%
First we show the following.
\begin{theorem}\label{th:0}
Let $Z(s)$ defined by (\ref{eq:dergd1}), $a(1) \ge 0$, $a(n) \ne 0 $ for some $n \in {\mathbb{N}}$ and $a(n) = O(n^\varepsilon)$ for any $\varepsilon >0$. Then for $\sigma >1$,
\begin{equation}\label{eq:Zsdef1}
{\mathcal{Z}}_\sigma (t) := \frac{Z(\sigma+{\rm{i}}t)}{Z(\sigma)}
\end{equation}
is a characteristic function if and only if $a(n) \ge 0$ for all $n \in {\mathbb{N}}$.
\end{theorem}

Hence, our basic assumption is the following.
\[
\mbox{(A)} \qquad \qquad \quad a(1) >0, \qquad a(n) \ge 0, \quad n \ge 2, \qquad 
a(n) = O(n^\varepsilon), \quad \forall \varepsilon >0. \qquad \qquad \qquad \quad
\]
By the theorem above, if $Z(s)$ satisfies the assumption (A), the function ${\mathcal{Z}}_\sigma (t)$ is a characteristic function for all $\sigma >1$. 

To state the main theorem, recall some facts on arithmetical functions. A complex-valued function defined on the positive integers is called arithmetical function (see e.g.~\cite[Chapter 2]{Apo}). If $a$ and $b$ are two arithmetic functions, one defines the Dirichlet convolution of $a$ and $b$ to be the arithmetic function $c$ given by
\begin{equation}\label{eq:defcn}
c(n) = (a * b) (n) := \sum_{d \mid n} a(d) \, b \Bigl( \frac{n}{d} \Bigr) =
\sum_{\alpha \beta = n} a(\alpha) b (\beta), 
\end{equation}
where the sum extends over all positive divisors $d$ of $n$, or equivalently, over all distinct pairs $(\alpha, \beta)$ of positive integers whose product is $n$. The arithmetic function $I$ defined by
\[
I(n) = 
\begin{cases}
1 & n=1,\\
0 & \mbox{otherwise},
\end{cases}
\]
is called the identity function. If $a$ is an arithmetical function with $a(1) \ne 0$, there exists a unique arithmetical function $a^{-1}$, called the Dirichlet inverse of $a$ such that
\[
a * a^{-1} = a^{-1} * a =I.  
\]
Furthermore, $a^{-1}$ is given by the following recursion formulas
\begin{equation}\label{eq:rec1}
a^{-1} (1) = \frac{1}{a(1)}, \qquad 
a^{-1} (n) = \frac{-1}{a(1)} \sum_{d \mid n, \, d <n} a^{-1}(d) \, a \Bigl( \frac{n}{d} \Bigr), \quad n>1
\end{equation}
(see \cite[Theorem 2.8]{Apo}). Now we put $a^\# (n) := a(n) \log n$ and
\begin{equation}\label{eq:A}
A(n) := \bigl( a^\#  * a^{-1} \bigr) (n).
\end{equation}
Note that this sequence $\{ A(n) \}$ plays important role in the main theorem and Theorem \ref{th:EV} below. We have the following which implies that any zeta distributions defined by (\ref{eq:Zsdef1}) with the assumption (A) is pretended infinitely divisible when $\sigma >1$ is sufficiently large. Moreover, if ${\mathcal{Z}}_\sigma (t)$ is infinitely divisible for some $\sigma_{id} >1$, then ${\mathcal{Z}}_\sigma (t)$ is infinitely divisible for all $\sigma >1$. We remark that the corresponding L\'evy or quasi-L\'evy measures are given by (\ref{eq:Levym1}) below.

\begin{mthm}\label{th:p1}
Suppose the assumption $(A)$. Then, one of (1) or (2) holds.

$(1)\,$ Assume ${\mathcal{Z}}_{\sigma} (t) =0$ for some $\sigma >1$ and  $t\in {\mathbb{R}}$. Then, there exist real numbers $\sigma_0 >1$ and $t_0\in {\mathbb{R}}$ such that ${\mathcal{Z}}_{\sigma_0} (t_0) =0$ and ${\mathcal{Z}}_{\sigma} (t) \ne 0$ for all $\sigma > \sigma_0$ and $t \in {\mathbb{R}}$. Moreover, ${\mathcal{Z}}_\sigma (t)$ is not infinitely but pretended infinitely divisible for all $\sigma >\sigma_0$, quasi infinitely divisible for all $\sigma >\sigma_0+1$, and not pretended infinitely divisible on the line $\sigma = \sigma_0$. Moreover, when $\sigma >\sigma_0+1$, the quasi L\'evy measure is finite and written by 
\begin{equation}\label{eq:Levym1}
N_\sigma (dx) := \sum_{n=2}^\infty \frac{A(n)}{n^{\sigma} \log n} \delta_{-\log n} (dx).
\end{equation}

$(2)\,$ Assume ${\mathcal{Z}}_{\sigma} (t) \ne 0$ for all $\sigma >1$ and  $t\in {\mathbb{R}}$. Then, one of (2-1) or (2-2) holds. \\
$\,(2\mbox{-}1)\,$ If $A(n) <0$ for some $n \in {\mathbb{N}}$, then ${\mathcal{Z}}_\sigma (t)$ is not infinitely but pretended infinitely divisible for all $\sigma >1$, and quasi infinitely divisible for all $\sigma >2$. Furthermore, the quasi L\'evy measure is finite and expressed as (\ref{eq:Levym1}) when $\sigma >2$. \\
$\, (2\mbox{-}2)\,$ If $A(n) >0$ for all $n \ge 1$, then ${\mathcal{Z}}_\sigma (t)$ is infinitely divisible for all $\sigma >1$. More precisely, ${\mathcal{Z}}_\sigma (t)$ is a compound Poisson characteristic function with its finite L\'evy measure $N_\sigma$ on ${\mathbb{R}}$ given by (\ref{eq:Levym1}).
\end{mthm}

\begin{remark}
In \cite[Theorem 8.1]{LS17}, it is shown that a distribution on ${\mathbb{R}}$ concentrated on the integers is quasi-infinitely divisible if and only if its characteristic function does not have zeros. The key of the proof of this statement is the Wiener theorem (see \cite[p.~118]{QQ}) which says that if $f(t) = \sum_{n \in {\mathbb{Z}}} \alpha_n e^{{\rm{i}}nt}$ with $\sum_{n \in {\mathbb{Z}}} |\alpha_n| < \infty$ and $f(t) \ne 0$ for all $t \in {\mathbb{R}}$, then it holds that
\[
\frac{1}{f(t)} = \sum_{n \in {\mathbb{Z}}} \beta_n e^{{\rm{i}}nt} 
\quad \mbox{with} \quad \sum_{n \in {\mathbb{Z}}} |\beta_n| < \infty .
\]
The analog of the Wiener theorem for Dirichlet series proved by Hewitt and Williamson is stated as follows (see \cite[Theorem 4.4.3]{QQ}). Let $f(s) = \sum_{n=1}^\infty \alpha_n n^{-s}$ with $\sum_{n=1}^\infty |\alpha_n| < \infty$ and assume $|f(s)| \ge \delta > 0$ for all $\Re(s) >0$. Then, for $\Re(s) >0$, one has 
\[
\frac{1}{f(s)} = \sum_{n=1}^\infty \beta_n n^{-s}
\quad \mbox{with} \quad \sum_{n=1}^\infty |\beta_n| < \infty .
\]
It should be noted that the key of the proof the main theorem is not the Hewitt and Williamson theorem above but Lemma \ref{lem:apo2} which is a corrected version of \cite[Theorem 11.14]{Apo} in Apostol's famous textbook (see also Section 2.2). 
\end{remark}

Let $X_\sigma$ be a random variable whose characteristic function is ${\mathcal{Z}}_{\sigma} (t)$. Then, by using $A(n)$ defined as (\ref{eq:A}), we have the following series expression of expectation ${\mathbb{E}} [X_\sigma]$ and the variance ${\rm{Var}} [X_\sigma]$, which should be compared with (\ref{eq:Levym1}). 
\begin{theorem}\label{th:EV}
Assume ${\mathcal{Z}}_{\sigma} (t)$ does not vanish when $\sigma > \sigma_0$. Then, for $\sigma > \sigma_0$, we have
\begin{equation}\label{eq:EV1}
{\mathbb{E}} [X_\sigma] = - \sum_{n=2}^\infty \frac{A(n)}{n^{\sigma}}, \qquad
{\rm{Var}} [X_\sigma] = \sum_{n=2}^\infty \frac{A(n) \log n}{n^{\sigma}} .
\end{equation}
\end{theorem}

In Section 2, we amend \cite[Theorem 11.14]{Apo} and give its proof. In Section 3, we prove the theorems above and give some examples satisfying (1), (2-1) and (2-2).

%%%%%%%%%%%%%%%%%%%%%%%%%%%%%%%%%%%%%
\section{Some facts of Number theory}
%%%%%%%%%%%%%%%%%%%%%%%%%%%%%%%%%%%%%%

%%%%%%%%%%%%%%%%%%%%%%%%%%%%%%%%%%%%%%%%%
\subsection{Lemmas}
%%%%%%%%%%%%%%%%%%%%%%%%%%%%%%%%%%%%%%%%%
We quote the following seven lemmas with some modification from \cite[Chapter 11]{Apo}. It should be noted that \cite[Theorem 11.14]{Apo} (see Lemma \ref{lem:apo2}) and its proof seem to be wrong. A corrected proof is written in Section 2.2.

\begin{lemma}[{see \cite[Lemma 3, in p.~234]{Apo}}]\label{lem:d1}
Let $\{f_n \}$ be a sequence of analytic functions on an open subset $S$ of the complex plane, and assume that $\{ f_n \}$ converges uniformly on every compact subset of $S$ to a limitation function $f$. Then $f$ is analytic on $S$. 
\end{lemma}

\begin{lemma}[{see \cite[Theorem 11.1]{Apo}}]\label{lem:absciabco}
Assume that the series $\sum_{n=1}^\infty |a(n)n^{-s}|$ does not converge for all $s \in {\mathbb{C}}$ or diverge for all $s \in {\mathbb{C}}$. Then there exists a real number $\sigma_a$, called the abscissa of absolute convergence, such that the series $\sum_{n=1}^\infty a(n)n^{-s}$ converges absolutely when $\sigma >\sigma_a$ and diverges when $\sigma < \sigma_a$.
\end{lemma}

\begin{lemma}[{see \cite[Theorem 11.4]{Apo}}]\label{lem:apo0}
Let $Z(s) = \sum_{n=1}^\infty a(n) n^{-s}$ converge absolutely when $\sigma >1$ and suppose that $Z(s) \ne 0$ for some $s \in {\mathbb{C}}$ with $\sigma >1$. Then there is a half-plane $\sigma > \gamma \ge 1$ in which $Z(s)$ is never zero.
\end{lemma}

\begin{lemma}[{see \cite[Theorem 11.9]{Apo}}]\label{lem:abscico}
If the series $\sum_{n=1}^\infty a(n)n^{-s}$ does not converge everywhere or diverge everywhere, then there exists a real number $\sigma_c$, called the abscissa of convergence, such that the series converges for all $s$ in the half-plane  $\sigma >\sigma_c$ and diverges for all $\sigma < \sigma_c$.
\end{lemma}

\begin{lemma}[{see \cite[Theorem 11.11]{Apo}}]\label{lem:d2}
A Dirichlet series $\sum_{n=1}^\infty a(n)n^{-s}$ converges uniformly on every compact subset lying interior to the half-plane of convergence $\sigma > \sigma_c$. 
\end{lemma}

\begin{lemma}[{see \cite[Theorem 11.13]{Apo}}]\label{lem:apo1}
Assume that $a(n) \ge 0$ for all $n \ge n_0$. If the series $Z(s) = \sum_{n=1}^\infty a(n) n^{-s}$ has a finite abscissa of convergence $\sigma_c$, then $Z(s)$ has a singularity on the real axis at the point $s=\sigma_c$.
\end{lemma}

\begin{lemma}[{see \cite[Theorem 11.14]{Apo}}]\label{lem:apo2}
Let $Z(s) = \sum_{n=1}^\infty a(n) n^{-s}$ converge absolutely when $\sigma >1$ and suppose that $a(1) \ne 0$. If $Z(s) \ne 0$ for $\sigma > \sigma_0 \ge 1$, then for $\sigma > \sigma_0$, 
\begin{equation}\label{eq:g1}
Z(s) = \exp \bigl( G(s) \bigr), \qquad
G(s) := \log a(1) + \sum_{n=2}^\infty \frac{A(n)}{\log n} n^{-s}, 
\end{equation}
where $A(n)$ is given by (\ref{eq:A}). Note that for complex $z \ne 0$, $\log z$ denotes that branch of the logarithm which is real when $z>0$ and the Dirichlet series $\sum_{n=2}^\infty A(n) n^{-s}/\log n$ converges uniformly on every compact subset in the half-plane $\sigma > \sigma_0$. Furthermore, the series for $G(s)$ converges absolutely when $\sigma > 1+\sigma_0$.
\end{lemma}

Next we show the three number theoretical lemmas below to prove Theorem \ref{th:0} and the Main theorem. First we show the following lemma.
\begin{lemma}\label{lem:abcon1}
Let $Z(s)$ defined by (\ref{eq:dergd1}) and suppose $a(n) = O(n^\varepsilon)$ for any $\varepsilon >0$. Then the Dirichlet series $Z(s)$ converges absolutely when $\sigma>1$. 
\end{lemma}
\begin{proof}
By the assumption that $a(n) = O(n^\varepsilon)$ for any $\varepsilon >0$, one has
\[
|Z(s)| \le \sum_{n=1}^\infty \frac{|a(n)|}{n^\sigma} \le 
C_\varepsilon \int_1^\infty \!\! \frac{dx}{x^{\sigma-\varepsilon}},
\]
for some positive constant $C_\varepsilon$. The integral above converges absolutely when $\sigma >1$. 
\end{proof}

By using the lemma above, we have the following.
\begin{lemma}\label{lem:apo0a}
Let $a(1) > 0$ and $a(n) = O(n^\varepsilon)$. Then there exists a half-plane $\sigma > \gamma \ge 1$ in which $Z(s)$ is never zero.
\end{lemma}
\begin{proof}
From Lemma \ref{lem:abcon1} and the assumption, the Dirichlet series $\sum_{n=1}^\infty a(n) n^{-s}$ converge absolutely when $\sigma >1$. Hence by Lebesgue's dominated convergence theorem, one has
\[
\lim_{\sigma \to \infty} \sum_{n=1}^\infty \frac{a(n)}{n^s} = 
a(1) + \sum_{n=2}^\infty \biggl( \lim_{\sigma \to \infty} \frac{a(n)}{n^s} \biggr) = a(1) + 0 >0
\]
since the summation above can be regarded as an integration on a discrete measure. Hence $Z(s) \ne 0$ for some $s \in {\mathbb{C}}$ with $\Re (s) >1$ from the inequality above. Therefore, we have Lemma \ref{lem:apo0a} by Lemma \ref{lem:apo0}. 
\end{proof}

The next lemma plays important role in the proof of the case (1) of Theorem \ref{th:p1}.
\begin{lemma}\label{lem:c1nt}
Let $a(n)$ satisfy the assumption (A) and suppose $A(n) \ge 0$ for all $n \in {\mathbb{N}}$. Then the series appeared in (\ref{eq:g1}) converges absolutely when $\sigma >1$. 
\end{lemma}
\begin{proof}
Suppose contrary that for $\sigma_a > 1$, the series $\sum_{n=2}^\infty (A(n)/\log n) n^{-s}$ converges absolutely when $\sigma > \sigma_a$ but does not converge absolutely when $\sigma < \sigma_a$. Obviously, one has $A(n)/\log n \ge 0$ for all $2 \le n \in {\mathbb{N}}$. Thus $G(s)$ and $Z(s)$ has a singularity at the point $s=\sigma_a$ form Lemma \ref{lem:apo1}. On the other hand, the Dirichlet series of $Z(s)$ converges absolutely when $\sigma >1$ by Lemma \ref{lem:abcon1}. Hence $Z(s)$ is analytic in the half-plane $\sigma >1$ by Lemmas \ref{lem:d1} and \ref{lem:d2}. This contradicts our assumption mentioned above that $Z(s)$ has a singularity at the point $s=\sigma_a>1$. 
\end{proof}

%%%%%%%%%%%%%%%%%%%%%%%%%%%%%%%%%%%%%%%%%%
\subsection{Proof of Lemma \ref{lem:apo2}}
%%%%%%%%%%%%%%%%%%%%%%%%%%%%%%%%%%%%%%%%%%
We quote the the following three lemmas in order to show Lemma \ref{lem:apo2}. Recall that $\sigma_a$ and $\sigma_c$ are given in Lemmas \ref{lem:absciabco} and \ref{lem:abscico}, respectively.

\begin{lemma}[{see \cite[Theorem 11.10]{Apo}}]\label{lem:apo1110}
For any Dirichlet series with finite $\sigma_c$, we have
\[
0 \le \sigma_a - \sigma_c \le 1 .
\]
\end{lemma}

\begin{lemma}[{see the first note in \cite[p.~122]{Ten}}]\label{lem:conv1}
Let $a$ and $b$ be two arithmetic functions and assume that $\sum_{n=1}^\infty a_n$ converges, $\sum_{n=1}^\infty |b_n| < \infty$ and $c_n$ is given by (\ref{eq:defcn}). Then $\sum_{n=1}^\infty c_n$ converges to $(\sum_{n=1}^\infty a_n) (\sum_{n=1}^\infty b_n)$.
\end{lemma}

\begin{lemma}[{see \cite[Satz 12]{Lan}}]\label{lem:Lan4}
Suppose that $\alpha_1 \ne 0$, $f(s) = \sum_{n=1}^\infty \alpha_n n^{-s}$ converges and $f(s) \ne 0$ for $\sigma > \alpha$. Then the series $1/f(s) = \sum_{n=1}^\infty \beta_n n^{-s}$ converges for $\sigma > \alpha$.
\end{lemma}

\begin{proof}[Proof of Lemma \ref{lem:apo2}]
The series $\sum_{n=1}^\infty a^{-1}(n) n^{-s}$ converges if $\sigma > \sigma _0$ from Lemma \ref{lem:Lan4} (see also \cite[Theorem 2]{BG} or the final remark in \cite[p.~127]{Ten}). Hence, the Dirichlet series $\sum_{n=1}^\infty a^{-1}(n) n^{-s}$ converges absolutely when $\sigma > 1+\sigma _0$ according to Lemma \ref{lem:apo1110}.

We can write $Z(s) = \exp (G(s))$ for some function $G(s)$ which is analytic for $\sigma > \sigma_0$ since $G(s) \ne 0$ in this half-plane. Then we have
\[
Z'(s) = G'(s) \exp \bigl( G(s) \bigr) = G'(s) Z(s),
\]
so $G'(s) = Z'(s) / Z(s)$. On the other hand, one has
\[
Z'(s) =  - \sum_{n=1}^\infty \frac{a(n) \log n}{n^s}, \qquad 
\frac{1}{Z(s)} = \sum_{n=1}^\infty \frac{a^{-1}(n)}{n^s}. 
\]
Note that the Dirichlet series of $Z'(s)$ converges absolutely when $\sigma> \sigma _0 \ge \sigma _a =1$ since one has $\log n = O(n^\varepsilon)$. Thus it holds that
\begin{equation}\label{eq:G}
G'(s) = \frac{Z'(s)}{Z(s)} = - \sum_{n=2}^\infty \frac{(a^\#*a^{-1}) (n)}{n^s}, 
\end{equation}
where $a^\#(n) = a(n) \log n$. The Dirichlet series above converges when $\sigma > \sigma _0$ from Lemma \ref{lem:conv1}. Integration gives
\[
G(s) = C + \sum_{n=2}^\infty \frac{A(n)}{\log n} n^{-s}, 
\]
where $C$ is a constant and $A(n) := (a^\#*a^{-1}) (n)$ is defined by (\ref{eq:A}). This termwise integration is justified by Lemma \ref{lem:d2}. Thus, the Dirichlet series of $G(s)$ converges uniformly on every compact subset in the half-plane $\sigma > \sigma_0$ from Lemma \ref{lem:d2}. Furthermore, the series of $G(s)$ converges absolutely when $\sigma > 1+ \sigma_0$ by Lemma \ref{lem:apo1110}. Letting $\sigma \to \infty$ we can find $G(\sigma+ {\rm{i}}t) \to C$ hence we have
\[
a(1) = \lim_{\sigma \to \infty} Z (\sigma+ {\rm{i}}t) = e^C
\]
so $C = \log a(1)$. This completes the proof. 
\end{proof}

\begin{remark}
In the proof above of Lemma \ref{lem:apo2}, we show that the series $\sum_{n=1}^\infty a^{-1}(n) n^{-s}$ converges when $\Re (s) > \sigma _0$ and converges absolutely if $\Re (s) > 1+\sigma _0$. In the proof of \cite[Theorem 11.14]{Apo}, it seems that the author of \cite{Apo} believes the series $\sum_{n=1}^\infty a^{-1}(n) n^{-s}$ converges absolutely when $\Re (s) > \sigma _0$ (see the last sentence of the proof of \cite[Theorem 11.14]{Apo}). However, this statement seems to be too stronger than the fact proved by Hewitt and Williamson (see the Remark in Section 1.3) and the following Landau-Schnee theorem (see \cite[Theorems 1 and 2]{BG}, \cite[S\"atze 5 and 12]{Lan} or \cite[Theorem 4.5.2]{QQ}). Assume that $f(s) = \sum_{n=1}^\infty \alpha_n n^{-s}$ with convergence and non-vanishing of $f$ for $1 < \sigma \le \infty$, so that in particular $\alpha_1 \ne 0$ and $\alpha_n = O (n^{1+\varepsilon})$, for any $\varepsilon >0$. Then it holds that
\[
\sigma >1 \Longrightarrow \frac{1}{f(s)} = \sum_{n=1}^\infty \beta_n n^{-s} \quad \mbox{with} \quad
\beta_n = O (n^{1+\varepsilon}), \quad \forall \varepsilon >0.
\]
\end{remark}

%%%%%%%%%%%%%%%%%%%%%%%%%%%%%%%%%%%%%%%%
\section{Proofs of theorems}
%%%%%%%%%%%%%%%%%%%%%%%%%%%%%%%%%%%%%%%%

%%%%%%%%%%%%%%%%%%%%%%%%%%%%%%%%%%%%%%%%
\subsection{Proofs of Theorems \ref{th:0} and \ref{th:EV}}
%%%%%%%%%%%%%%%%%%%%%%%%%%%%%%%%%%%%%%%%
Let $a(n) \ge 0$ for all $n \in {\mathbb{N}}$, $a(n) \ne 0 $ for some $n \in {\mathbb{N}}$ and $a(n) = O(n^\varepsilon)$ for any $\varepsilon >0$. Then we define a generalized Dirichlet $L$ random variable $X_{\sigma}$ with probability distribution on ${\mathbb{R}}$ given by
\begin{equation}\label{eq:defpr1}
{\rm{Pr}} \bigl(X_{\sigma}= - \log n \bigr)= \frac{a(n)n^{-\sigma}}{Z(\sigma)}.
\end{equation}
It should be noted that the distribution above is  a special case of multidimensional Shintani zeta distribution defined by Aoyama and Nakamura \cite{AN12s} (see also \cite[Definition 2.1]{Nakamura12}). It is easy to see that these distributions are probability distributions since one has $a(n)n^{-\sigma}/Z(\sigma) \ge 0$ for all $n \in {\mathbb{N}}$ and 
\[
\sum_{n=1}^{\infty} \frac{a(n)n^{-\sigma}}{Z(\sigma)} = \frac{1}{Z(\sigma)} 
\sum_{n=1}^\infty \frac{a(n)}{n^{\sigma}}=\frac{Z(\sigma)}{Z(\sigma)}=1 .
\]
\begin{lemma}\label{lem:1}
Let $X_{\sigma}$ be a generalized Dirichlet $L$ random variable. Then its characteristic function ${\mathcal{Z}}_{\sigma}$ is given by
\begin{equation}\label{eq:cf1}
{\mathcal{Z}}_\sigma (t) := \frac{Z(\sigma+{\rm{i}}t)}{Z(\sigma)}, \qquad t \in {\mathbb{R}}.
\end{equation}
\end{lemma}
\begin{proof}
Form the definition, we have, for any $t \in {\mathbb{R}}$, 
\[
{\mathcal{Z}}_{\sigma}(t) = \sum_{n=1}^\infty e^{{\rm i}t (- \log n)} \frac{a(n)n^{-\sigma}}{Z(\sigma)} =
\frac{1}{Z(\sigma)} \sum_{n=1}^\infty \frac{e^{-{\rm i}t \log n}a(n)}{n^{\sigma}} =
\frac{1}{Z(\sigma)} \sum_{n=1}^\infty  \frac{a(n)}{n^{\sigma+{\rm{i}}t}} =
\frac{Z(\sigma +{\rm i}t)}{Z(\sigma)}.
\]
Note that the series $\sum_{n=1}^\infty a(n) n^{-\sigma}$ converges absolutely when $\sigma >1$ by Lemma \ref{lem:abcon1}. The equality above implies the lemma.
\end{proof}
Now we only have to show the following lemma (see also \cite[Lemma 2.3]{Nakamura12}) in order to prove Theorem \ref{th:0}.
\begin{lemma}\label{lem:cf2}
Assume that there exists $m \in {\mathbb{N}}$ such that $a(m)<0$. Then the function ${\mathcal{Z}}_{\sigma}(t)$ is not a characteristic function. 
\end{lemma}
\begin{proof}
First suppose $Z(\sigma) =0$. In this case, the function ${\mathcal{Z}}_{\sigma}(t) = Z(\sigma +{\rm i}t)/Z(\sigma)$ is not a characteristic function. 

Assume $Z(\sigma) \ne 0$. Let ${\mathbb{N}}_+$ be the set of integers $n$ such that $a(n)\ge 0$ and ${\mathbb{N}}_-$ be the set of integers $m$ such that $a(m)< 0$, respectively. From the view of (\ref{eq:defpr1}), we have 
\begin{equation*}
\begin{split}
& \frac{Z(\sigma +{\rm i}t)}{Z(\sigma)} = \frac{1}{Z(\sigma)} \sum_{n=1}^\infty \frac{a(n)}{n^{\sigma+{\rm i}t}}= \frac{1}{Z(\sigma)} \sum_{m \in {\mathbb{N}}_-} \frac{a(m)}{m^{\sigma+{\rm i}t}} + 
\frac{1}{Z(\sigma)} \sum_{n \in {\mathbb{N}}_+} \frac{a(n)}{n^{\sigma+{\rm i}t}} \\  = &
\frac{1}{Z(\sigma)} \sum_{m \in {\mathbb{N}}_-} \frac{a(m)}{m^{\sigma}}
\int_{\mathbb{R}} e^{{\rm i}tx} \delta_{- \log m} (dx) + 
\frac{1}{Z(\sigma)} \sum_{n \in {\mathbb{N}}_+} \frac{a(n)}{n^{\sigma}}
\int_{\mathbb{R}} e^{{\rm i}tx} \delta_{- \log n} (dx) .
\end{split}
\end{equation*}
From the definition of ${\mathbb{N}}_+$ and ${\mathbb{N}}_-$, 
\begin{equation}
\frac{1}{Z(\sigma)} \sum_{m \in {\mathbb{N}}_-} \frac{a(m)}{m^{\sigma}} \delta_{- \log m} + 
\frac{1}{Z(\sigma)} \sum_{n \in {\mathbb{N}}_+} \frac{a(n)}{n^{\sigma}}
\delta_{- \log n} 
\label{me:1}
\end{equation}
is not a measure but a signed measure. Moreover, we have
\[
\frac{1}{|Z(\sigma)|} \int_{\mathbb{R}}
\sum_{n=1}^\infty \left| \frac{a(n)}{n^{\sigma}} \right| \delta_{- \log n} (dx) = 
\frac{1}{|Z(\sigma)|} \sum_{n=1}^\infty \frac{|a(n)|}{n^{\sigma}} < \infty
\]
by Lemma \ref{lem:abcon1}. Hence the signed measure (\ref{me:1}) has finite total variation. It is well-known that any signed measure with finite total variation is uniquely determined by the Fourier transform. Therefore, $Z(\sigma +{\rm i}t)/Z(\sigma)$ is not a characteristic function.
\end{proof}

\begin{proof}[Proof Theorem \ref{th:EV}]
We can easily see that
\[
Z' (\sigma) = - \sum_{n=2}^\infty \frac{a(n) \log n}{n^\sigma}, \qquad Z'' (\sigma) = \sum_{n=2}^\infty \frac{a(n) \log^2 n}{n^\sigma}.
\]
Hence, we obtain
\begin{equation*}
\begin{split}
&{\mathbb{E}} [X_\sigma] =  \frac{1}{Z(\sigma)} \sum_{n=2}^\infty (-\log n) \frac{a(n)}{n^\sigma} = \frac{Z'(\sigma)}{Z(\sigma)}, \\
&{\rm{Var}} [X_\sigma] = {\mathbb{E}} [X_\sigma^2] - \bigl( {\mathbb{E}} [X_\sigma] \bigr)^2 = 
\frac{1}{Z(\sigma)} \sum_{n=2}^\infty \frac{a(n) \log^2 n}{n^{\sigma}} - \Bigl( \frac{Z'(\sigma)}{Z(\sigma)} \Bigr)^2 \\
&= \frac{Z''(\sigma)}{Z(\sigma)} - \Bigl( \frac{Z'(\sigma)}{Z(\sigma)} \Bigr)^2 = \frac{Z''(\sigma)Z(\sigma) - Z'(\sigma)^2}{Z^2(\sigma)}
= \frac{d}{d\sigma} \frac{Z'(\sigma)}{Z(\sigma)} .
\end{split}
\end{equation*}
Therefore, we have (\ref{eq:EV1}) by (\ref{eq:A}) and (\ref{eq:G}) when $Z(s) \ne 0$ for $\sigma > \sigma_0$.
\end{proof}

%%%%%%%%%%%%%%%%%%%%%%%%%%%%%%%%%%%%%%%%%%%%%%%%
\subsection{The case (2-2) of Main Theorem}
%%%%%%%%%%%%%%%%%%%%%%%%%%%%%%%%%%%%%%%%%%%%%%%%
In this subsection, we suppose the assumption (A) and $A(n) \ge 0$ for all $n \in {\mathbb{N}}$.
\begin{proof}[Proof of Main Theorem (2-2)]
By the assumption (A), Lemma \ref{lem:abcon1} and Theorem \ref{th:0}, ${\mathcal{Z}}_\sigma (t)$ defined by (\ref{eq:Zsdef1}) is a characteristic function for all $\sigma >1$. From Lemma \ref{lem:apo0}, there exists a half-plane $\sigma > \gamma \ge 1$ in which $Z(s)$ does not vanish. Hence for $\sigma > \gamma$, we have
\[
{\mathcal{Z}}_\sigma (t) = \frac{Z(\sigma+{\rm{i}}t)}{Z(\sigma)} =
\exp \Biggl( \sum_{n=2}^\infty \frac{A(n)}{n^{\sigma} \log n} \bigl( n^{-{\rm{i}}t} -1 \bigr) \Biggr) 
\]
according to Lemma \ref{lem:apo2}. Moreover, the Dirichlet series in the formula above converges absolutely when $\sigma >1$ by Lemma \ref{lem:c1nt} and the assumption $A(n) \ge 0$ for all $n \in {\mathbb{N}}$. Therefore, for all $\sigma >1$, it holds that
\begin{equation}\label{eq:zp1}
{\mathcal{Z}}_\sigma (t) =
\exp \Biggl( \sum_{n=2}^\infty \frac{A(n)}{n^{\sigma} \log n} \bigl( n^{-{\rm{i}}t} -1 \bigr) \Biggr) =
\exp \biggl[ \int_{\mathbb{R}} \bigl( e^{{\rm{i}}tx} -1 \bigr) N_{\sigma}(dx) \biggr],
\end{equation}
where $N_{\sigma}(dx)$ is defined by (\ref{eq:Levym1}). In addition, we have
\[
N_{\sigma}({\mathbb{R}}) \le 
\int_{\mathbb{R}} \sum_{n=2}^\infty \frac{A(n)}{n^{\sigma} \log n} \delta_{-\log n} (dx)
= \sum_{n=2}^\infty \frac{A(n)}{n^{\sigma} \log n} < \infty.
\]
Hence ${\mathcal{Z}}_\sigma (t)$ is a compound Poisson characteristic function with a finite L\'evy measure expressed as (\ref{eq:Levym1}) for all $\sigma >1$. 
\end{proof}

We have the following examples satisfy the assumption (A) and $A(n) \ge 0$ for all $n \in {\mathbb{N}}$.
\begin{example}[the Riemann zeta distribution, see Proposition {\ref{pro:RD}}]
Let $a(n)=1$ for all $n \in {\mathbb{N}}$. Then it holds that
\[
\frac{A(n)}{\log n} = 
\begin{cases}
1/r & \exists p \in {\mathbb{P}}, \,\, \exists r \in {\mathbb{N}} \,\mbox{ s.t. } n=p^r,\\
0 & \mbox{otherwise},
\end{cases}
\]
where ${\mathbb{P}}$ is the set of prime numbers.
\end{example}

\begin{example}[see {\cite[Theorem 2]{Lin}}]
Let $a(1)=1$, $a(n) = O(n^\varepsilon)$ and $a(n) \ge 0$ be completely multiplicative for all $n \in {\mathbb{N}}$. Then one has
\[
\frac{A(n)}{\log n} = 
\begin{cases}
a(p)^r/r & \exists p \in {\mathbb{P}}, \,\, \exists r \in {\mathbb{N}} \,\mbox{ s.t. } n=p^r,\\
0 & \mbox{otherwise},
\end{cases}
\]
(see also \cite[Theorem 4]{Lin}). For instance, we can take $a(n)=n^\alpha$, $\alpha \le 0$. 
\end{example}

\begin{example}
Let $a(n) = d_k (n)$, where $d_k(n)$, $k=2,3,4,\ldots$ denotes the number of ways of expressing $n$ as a product of $k$ factors, expression with the same factors in a different order being counted as different. Then we have
\[
\frac{A(n)}{\log n} = 
\begin{cases}
k/r & \exists p \in {\mathbb{P}}, \,\, \exists r \in {\mathbb{N}} \,\mbox{ s.t. } n=p^r,\\
0 & \mbox{otherwise}.
\end{cases}
\]
It is well-know that one has $d_k(n)= O (n^\varepsilon)$ (see for example \cite[(1.70)]{Ivic}). 
\end{example}

%%%%%%%%%%%%%%%%%%%%%%%%%%%%%%%%%%%%%%%%%%%
\subsection{The case (2-1) of Main Theorem}
%%%%%%%%%%%%%%%%%%%%%%%%%%%%%%%%%%%%%%%%%%%
In this subsection, suppose the assumption (A), $A(n) < 0$ for some $n \in {\mathbb{N}}$ and $Z(s)$ does not vanish for all $\sigma >1$. 
\begin{proof}[Proof of Main Theorem (2-1)]
From the assumptions above and Lemma \ref{lem:apo2}, we have (\ref{eq:zp1}) for all $\sigma >1$. Hence the function ${\mathcal{Z}}_\sigma (t)$ defined by (\ref{eq:Zsdef1}) is a not infinitely divisible but pretended infinitely divisible characteristic function if $\sigma >1$. 

Furthermore, the series $\sum_{n=2}^\infty A(n) n^{-\sigma} / \log n$ converges absolutely when $\sigma >2$ by Lemma \ref{lem:apo2} and the assumption that $Z(s)$ does not vanish for all $\sigma >1$. Therefore, it holds that
\begin{equation}\label{eq:clmf}
|N_{\sigma}|({\mathbb{R}}) \le 
\int_{\mathbb{R}} \sum_{n=2}^\infty \frac{|A(n)|}{n^{\sigma} \log n} \delta_{-\log n} (dx)
= \sum_{n=2}^\infty \frac{|A(n)|}{n^{\sigma} \log n} < \infty
\end{equation}
when $\sigma >2$. Thus the quasi L\'evy measure expressed as (\ref{eq:Levym1}) is finite for all $\sigma >2$. In addition, ${\mathcal{Z}}_\sigma (t)$ is not infinitely but quasi infinitely divisible for all $\sigma >2$ by (\ref{eq:zp1}) and the assumption $A(n) < 0$ for some $n \in {\mathbb{N}}$.
\end{proof}

\begin{example}[see {\cite[Theorem 1.2]{Nakamura12}}]
Let $a(n)=1$ for $n=1,q$, where $q$ is a natural number and $a(n)=0$ otherwise. The function $1+q^{-s}$ does not vanish for all $\sigma >1$ since one has $|q^{-s}| <1$. By using 
\[
\log \bigl( 1 + q^{-s} \bigr) = \sum_{r=1}^\infty \frac{(-1)^{r-1}}{rq^{rs}},
\]
we can see that
\begin{equation}\label{qualevi:1}
\frac{A(n)}{\log n} = 
\begin{cases}
(-1)^{r-1}/r & \exists r \in {\mathbb{N}} \,\, \mbox{ s.t. } n= q^r,\\
0 & \mbox{otherwise}.
\end{cases}
\end{equation}
Obviously the series $\sum_{k=1}^\infty (-1)^{k-1} k^{-1} q^{-ks}$ converges absolutely when $\sigma >1$. 
\end{example}

\begin{example}
Let $a(n)=\mu (n)$, where $\mu(1)=1$, $\mu(n)=(-1)^k$ if $n$ is the product of $k$ different primes and $\mu(n)=0$ if $n$ contains any factor to a power higher than the first. Then it is known (see {\cite[(1.2.7)]{Tit}}) that
\[
\frac{\zeta (s)}{\zeta (2s)} = \sum_{n=1}^\infty \frac{|\mu(n)|}{n^s}, \qquad \sigma >1.
\]
Meanwhile, it holds that
\[
\frac{\zeta (s)}{\zeta (2s)} = \prod_p \frac{1-p^{-2s}}{1-p^{-s}} =
\prod_p \frac{(1-p^{-s})(1+p^{-s})}{1-p^{-s}} = \prod_p \bigl( 1+p^{-s} \bigr). 
\]
From (\ref{qualevi:1}) and the formula above, one has
\[
\frac{A(n)}{\log n} = 
\begin{cases}
(-1)^{r-1}/r & \exists p \in {\mathbb{P}}, \,\, \exists r \in {\mathbb{N}} \,\, \mbox{ s.t. } n= p^r,\\
0 & \mbox{otherwise}.
\end{cases}
\]
\end{example}
The function $\zeta (s)/\zeta (2s) \ne 0$ when $\sigma >1$ according to the Euler product of $\zeta (s)$. Furthermore, one has
\[
\sum_p \sum_{r=1}^\infty \biggl| \frac{(-1)^{r-1}}{r} p^{-rs} \biggr| \le 
\sum_p \sum_{r=1}^\infty \biggl| p^{-rs} \biggr| \le \sum_{n=1}^\infty n^{-\sigma} <\infty
\]
for $\sigma >1$. Thus the corresponding quasi L\'evy measure is finite for all $\sigma >1$. 

%%%%%%%%%%%%%%%%%%%%%%%%%%%%%%%%%%%%%%%%%
\subsection{The case (1) of Main Theorem}
%%%%%%%%%%%%%%%%%%%%%%%%%%%%%%%%%%%%%%%%%
In this subsection, suppose the assumption (A), $A(n) < 0$ for some $n \in {\mathbb{N}}$ and $Z(s)=0$ for some $\sigma > 1$ and $t \in {\mathbb{R}}$. 
\begin{proof}[Proof of Main Theorem (1)]
According to Lemma \ref{lem:apo0a}, there exist $\sigma_0 >1$ and $t_0 \in {\mathbb{R}}$ such that $Z(\sigma_0+ {\rm{i}}t_0) = 0$ and $Z(s) \ne 0$ for all $\sigma > \sigma_0$. Hence for all $\sigma > \sigma_0$, we have (\ref{eq:zp1}) by Lemma \ref{lem:apo2} and modifying the proof of Main Theorem (2-1). Thus ${\mathcal{Z}}_\sigma (t)$ is a pretended infinitely divisible characteristic function if $\sigma > \sigma_0$.  

Moreover, the corresponding quasi L\'evy measure expressed as (\ref{eq:Levym1}) is finite for all $\sigma > 1+\sigma_0$ from (\ref{eq:clmf}) and Lemma \ref{lem:apo2}. Hence ${\mathcal{Z}}_\sigma (t)$ is a not infinitely divisible but quasi infinitely divisible characteristic function when $\sigma > 1+\sigma_0$.  

Any pretended infinitely divisible characteristic function does not vanish from the representation (\ref{INF}) and the fact that $\exp (z) \ne 0$ for any $z \in {\mathbb{C}}$. Hence, ${\mathcal{Z}}_\sigma (t)$ is not pretended infinitely divisible when $\sigma = \sigma_0$ since one has $Z(\sigma_0+ {\rm{i}}t_0) = 0$ for some  $t_0 \in {\mathbb{R}}$.
\end{proof}

\begin{example}
Let $2H(s) = \zeta^2(s)+ \zeta (2s)$. Then it holds that
\[
H(s) = \sum_{n=1}^\infty \frac{a(n)}{n^s}, \qquad
a(n) =
\begin{cases}
1 & \exists m \in {\mathbb{N}} \,\, \mbox{ s.t. } n= m^2,\\
1/2 & \mbox{otherwise}.
\end{cases}
\]
The function $H(s)$ is the one variable version of the Euler-Zagier double zeta star function. Note that for any $\delta >0$, this function has infinitely many zeros in the strip $1<\sigma < 1+\delta$ from \cite[Corollary 1.6]{NP16}. In this case, we have $A(n)=\log n$ for $n= 2,3,5,7$, and
\begin{equation*}
\begin{split}
&A(4) = (7/8) \log 4 , \qquad A(6) = (1/4) \log 6, \\ & A(8) = (1/8) \log 8, \qquad A(12) = -(1/8) \log 12 < 0.
\end{split}
\end{equation*}
This is proved as follows. By the definitions of $a(n)$ and $a^{-1}(n)$, one has
\[
1= a^{-1} (1), \qquad 0= a(1) a^{-1} (2) + a(2) a^{-1} (1), \qquad 0= a(1) a^{-1} (3) + a(3) a^{-1} (1)
\]
which imply $a^{-1}(2) = a^{-1} (3) = - 1/2$. By (\ref{eq:A}), we have $A(2)=\log 2$, $A(3) = \log 3$ and
\[
A(4) = a^\#(4) a^{-1}(1) + a^\#(2) a^{-1}(2) + a^\# (1) a^{-1} (4) = \log 4 + (-1/2)(1/2) \log 2 = (7/8) \log 4. 
\]
From the recursion formulas (\ref{eq:rec1}), we obtain
\[
-a^{-1} (4) = a^{-1} (1) a(4) + a^{-1} (2) a(2) = 1 + (-1/2)(1/2) =3/4.
\]
Moreover, one has $a^{-1}(5) = a^{-1} (7) = - 1/2$, $A(5) = \log 5$, $A(7)= \log 7$ and
\[
a^{-1} (6) = a^{-1} (1) a(6) + a^{-1} (2) a(3) + a^{-1} (3) a(2)  = 1/2 + (-1/2)(1/2) + (-1/2)(1/2) =0
\]
which implies
\begin{equation*}
\begin{split}
A(6) &= a^\#(6) a^{-1}(1) + a^\#(3) a^{-1}(2) + a^\# (2) a^{-1} (3) + a^\# (1) a^{-1} (6)\\
&= (1/2) \log 6 + (1/2) (-1/2) \log 3 + (1/2) (-1/2) \log 2 + 0 = (1/4) \log 6. 
\end{split}
\end{equation*}
In addition, it holds that
\begin{equation*}
\begin{split}
A(8) &= a^\#(8) a^{-1}(1) + a^\#(4) a^{-1}(2) + a^\#(2) a^{-1}(4) + a^\# (1) a^{-1} (8) \\
&= (1/2) \log 8 + (-1/2) \log 4 + (-3/4)(1/2) \log 2 + 0 =(1/8) \log 2 ,
\end{split}
\end{equation*}
\begin{equation*}
\begin{split}
&A(12) = 
a^\#(12) a^{-1}(1) + a^\#(6) a^{-1}(2) + a^\#(4) a^{-1}(3) + a^\# (3) a^{-1} (4) + a^\#(2) a^{-1}(6) \\
&= (1/2) \log 12 + (-1/2)(1/2) \log 6 + (-1/2) \log 4 + (-3/4)(1/2) \log 3 = -(1/8) \log 12  .
\end{split}
\end{equation*}

Finally, we show $H(s)$ does not vanish when $\sigma >2$. Obviously, we have
\begin{equation*}
\begin{split}
& \bigl|H(s)\bigl| \ge 1 - \sum_{n=2}^\infty \frac{a(n)}{n^\sigma} \ge 1 - \sum_{n=2}^\infty \frac{1}{n^\sigma}
\ge 1 - 2^{-\sigma} - \int_2^\infty \frac{dx}{x^\sigma}\\
& \ge 1 - 2^{-\sigma} -\frac{2^{1-\sigma}}{\sigma-1} \ge 1 - \frac{1}{4} - \frac{1}{2} >0
\end{split}
\end{equation*}
since the function $2^{-x}/x$ is monotonically decreasing if $x>0$. Therefore, the zeta distribution defined by $H(s)$ is not infinitely but quasi infinitely divisible for all $\sigma >3$. 
\end{example}

%%%%%%%%%%%%%%%%%%%%%%%%%%%%%
\subsection*{Acknowledgments}
%%%%%%%%%%%%%%%%%%%%%%%%%%%%%
The author would like to thank Doctor Tomokazu Onozuka who gave some helpful comments. The author was partially supported by JSPS grant 16K05077. 

%%%%%%%%%%%%%%%%%%%%%%%%%%
 
\end{document}